\documentclass[11pt]{article}
\usepackage{amsmath, amssymb, amscd, amsthm, amsfonts}
\usepackage{graphicx}
\usepackage{mathrsfs}
\usepackage{hyperref}
\usepackage{color}
\usepackage{float} 
\usepackage{epstopdf}
\usepackage{tikz}%画图工具包
\usepackage{pgfplots}
\usepackage{subcaption}
\usepackage{caption}
\usepackage{xcolor}
\usetikzlibrary{arrows.meta}%画图
\usetikzlibrary{calc}%画图

\newtheorem{theorem}{Theorem}[section]

\newtheorem{claim}{Claim}
\newtheorem{subclaim}{Subclaim}[claim]

\newtheorem{observation}{Observation}

\title{Odd Induced Subgraphs in Graphs of Maximum Degree Four}
\author{Jiangdong Ai\thanks{School of Mathematical Sciences and LPMC, Nankai University. {\tt jd@nankai.edu.cn}.}
\hspace{2mm}
Qiwen Guo\thanks{Department of Computer Science, Royal Holloway University of London, {\tt gqwmath@163.com}.}
\hspace{2mm}
Gregory Gutin\thanks{Department of Computer Science, Royal Holloway University of London, {\tt g.gutin@rhul.ac.uk}, and School of Mathematical Sciences and LPMC, Nankai University.}
\hspace{2mm} Yiming Hao\thanks{School of Mathematical Sciences and LPMC, Nankai University. {\tt  
1120230031@mail.nankai.edu.cn}.}
\hspace{2mm} Anders Yeo\thanks{Department of Mathematics and Computer Science, University of Southern Denmark. {\tt yeo@imada.sdu.dk}, and Department of Mathematics, University of Johannesburg.}
}
\begin{document}
\maketitle

\begin{abstract}
%A graph is called {\em odd} if all of its vertex degrees are odd. A long-standing conjecture asked whether there exists a positive constant
%$c$ such that every $n$-vertex graph without isolated vertices contains an odd induced subgraph on at least $cn$ vertices. In 2022, Ferber and Krivelevich resolved this conjecture affirmatively with $c=10^{-4}.$
%A natural question is to determine the largest possible constant $c.$ 
%In 1994, Caro remarked that $c\le 2/7.$
%To the best of our knowledge, the bound $c\ge 2/7$ has not been improved. 
%Is $2/7$ the largest value of $c$? In this paper, we prove that every 
%$n$-vertex graph with maximum degree at most 4 and without isolated vertices contains an odd induced subgraph on at least 
%$2n/7$ vertices. Our result provides some support for the answer to the question above being positive.

A graph is called {\em odd} if all of its vertex degrees are odd. A long-standing conjecture asked whether there exists a positive constant
$c$ such that every $n$-vertex graph without isolated vertices contains an odd induced subgraph on at least $cn$ vertices. In 2022, Ferber and Krivelevich resolved this conjecture affirmatively with $c=10^{-4}.$
A natural question is to determine the largest possible constant $c.$ In 1994, Caro remarked that if $2/7$ is a valid value for $c$, then it is the largest possible one. To the best of our knowledge, the bound $c\ge 2/7$ has not been improved. %To the best of our knowledge, $2/7$ remains the largest known candidate for the optimal constant to date. 
Previous research has established tight bounds for specific graph classes—for instance, %$c = 2/3$ for trees and 
$c = 2/5$ for graphs with maximum degree at most $3$ and without isolated vertices.
In this paper, we prove that $c=2/7$ is the tight bound for graphs with maximum degree at most $4$ and without isolated vertices. Our result provides some support for $2/7$ being the largest value of $c$.

\end{abstract}
\noindent\textbf{Keywords:}  odd induced subgraphs, maximum degree, extremal graph theory.

%For graphs of maximum degree at most 3, Berman, Wang and Wargo proved that the optimal value is $c=2/5.$
%In this paper, we determine the corresponding optimal constant for graphs of maximum degree at most 4. We prove that every 
%$n$-vertex graph with maximum degree at most 4 and with no isolated vertices contains an odd induced subgraph on at least 
%$2n/7$ vertices, and that this bound is best possible. 
%In 1994, Caro remarked that $2/7$ was the largest known value of $c.$
%To the best of our knowledge, $2/7$ remains the largest known value of $c.$
%Does it mean $2/7$ is the largest value of $c$ over all graphs without isolated vertices? Our result provides some support that the answer is yes. 

\section{Introduction}
A graph is called {\em even} (respectively, {\em odd}) if every vertex has even (respectively, odd) degree.
By a classical theorem of Gallai, the vertex set of every graph can be partitioned into two parts, each inducing an even subgraph.
In particular, every $n$-vertex graph has an induced even subgraph on at least 
$n/2$ vertices. No such statement holds for odd induced subgraphs 
in general, due to the possible presence of isolated vertices.

A long-standing conjecture~\cite{Caro1994,Scott1992} stated that there is a constant $c>0$ such that every $n$-vertex graph without isolated vertices has an odd induced subgraph with at least $cn$ vertices. This conjecture attracted considerable attention. Caro~\cite{Caro1994} proved that every $n$-vertex graph without isolated vertices has an odd induced subgraph on at least $\Omega(\sqrt{n})$ vertices. This result was improved to $\Omega(n / \log n)$ by Scott~\cite{Scott1992}.
Recently, Ferber and Krivelevich~\cite{FK} confirmed the conjecture, proving that one can take $c = 10^{-4}$. This raises a natural question: what is the largest value of $c?$ (Formally, if $f_o(G)$ denotes the size of a largest induced odd subgraph in $G$, then we can define $c$ as the infimum of $f_o(G)/|V(G)|$ over all graphs $G$ without isolated vertices.)
%(Formally, we can define such a value as follows. Define $c_n$ as the maximum of $c$ such that every graph with at most $n$ vertices and without isolated vertices, has an odd induced subgraph with at least $c_n\cdot n$ vertices. Since $c_n$ is a monotonically decreasing sequence of real numbers bounded from below, it has a limit which we can define as the largest value of $c$.) 
Caro \cite{Caro1994} mentioned, as a known fact, that $c\le \frac{2}{7}.$ 
Are there $n$-vertex graphs without isolated vertices in which all odd induced subgraphs have less than $2n/7$ vertices? Theorem \ref{thm: Delta=4} gives some support for the negative answer.
%The answer seems likely to be yes, but the main result of this paper, Theorem \ref{thm: Delta=4}, gives some support for the opposite answer. Which answer is the correct one?  

For several graph classes, the optimal constant 
$c$ is already known. Radcliffe and Scott ~\cite{Scott1995} proved that every $n$-vertex tree has an odd induced subgraph of order at least $2\lfloor (n+1)/3 \rfloor$. 
Berman et~al.~\cite{BWW1997} established the optimal value $c=2/5$
for graphs of maximum degree 3, and Hou et al.~\cite{hou2018} showed that the same holds for graphs of treewidth at most two.
 Rao et~al.~\cite{Rao2022} proved that $c\geq \frac{2}{5}$ for planar graphs with girth at
least 7, and $c\geq \frac{1}{3}$ for planar graphs with girth at least 6, and that both bounds are tight.  

In this paper, we extend the degree-bound result of Berman et al.~\cite{BWW1997}. We determine the maximum constant $c$ such that every graph with maximum degree at most 4 and no isolated vertices contains an odd induced subgraph on at least $cn$ vertices.

\begin{theorem}\label{thm: Delta=4}
  Every graph of order $n$ without isolated vertices and with maximum degree at most four has an odd induced subgraph of order at least $2n/7$. Furthermore, this bound is sharp.
\end{theorem}

A related line of research concerns a conjecture of Scott~\cite{Scott1992} who proposed that every $n$-vertex graph with chromatic 
number $\chi$ and no isolated vertices contains an odd induced subgraph with at least $n/(2\chi)$ vertices. Recent work of Wang and Wu~\cite{T2024} disproved this conjecture for bipartite graphs, while confirming it for all line graphs; the conjecture remains open for graphs with $\chi\ge 3.$

In the next section, we prove Theorem \ref{thm: Delta=4}. We conclude the paper in Section \ref{sec:conc}.

\section{Proof of Theorem \ref{thm: Delta=4}}

Assume that the theorem is false and let $G$ be a counterexample to the theorem of the smallest possible order. Note that $G$ is connected. We will prove several claims concerning $G$, ending in a contradiction, which will complete the proof. Let $G$ have order $n$, $\Delta(G)\leq 4$, and have no isolated vertices. Note that when $n\le 7$, every edge forms an odd induced subgraph of order at least $2n\slash 7$, so we assume $n>7$.
%To prove this theorem, we proceed by induction on $n$. When $n\le 7$, every edge forms an odd induced subgraph of order at least $2n\slash 7$. Suppose that, for $n\ge 7$, every simple graph $G'$ of order less than $n$ without isolated vertices and with maximum degree at most four has an odd induced subgraph of order at least $2|V(G')|/7$.
%    Let $G$ be a graph of order $n$ with $\Delta(G)\leq 4$, and $G$ has no isolated vertex. Without loss of generality, assume that $G$ is connected.
    
%    Firstly, we claim that the result holds if $G$ has a vertex of degree one.
%We begin by showing that we may assume that the minimum degree of the graph is at least three.

    \begin{claim}\label{claim: delta(G)=1}
    $\delta(G)>1$.   
        \end{claim}
\begin{proof}
    For the sake of contradiction, assume that $\delta(G)=1$.
    Let $p\in V(G)$ be a vertex of degree one, and $N_G(p)=\{x\}$. Let  $G'=G- N_G[x]$ and $N_G(x)=\{p,y_1,\dots,y_{d_G(x)-1}\}$. 
    %If $G'$ has no isolated vertex, then the odd induced subgraph of order at least $2(|V(G)|-5)\slash 7$ which is contained in $G'$, together with the edge $px$ is an odd induced subgraph of order at least $2|V(G)|\slash 7$, which is a contradiction. Thus, 
    Let $I$ be the set of isolated vertices in $G'$ and let $G''=G'- I$. Then $G''$ contains an odd induced subgraph $H$ of order at least $2|V(G'')|\slash 7$ by the minimality of $n$. 
    
    Note that $|I\cup N_G[x]|\ge 8$, since otherwise $H$ together with the edge $px$ is an odd induced subgraph of order at least $2|V(G)|\slash 7$, which is a contradiction. Therefore, $d_G(x)\ge 3$ and $9\ge |I|\geq 3$.

     Note that  every $y_i$ which has two neighbors in $I$, must have one neighbor in $V(G'')$, since otherwise suppose that $\{z_1,z_2\}\subseteq N_G(y_1)\cap I$ and $y_1$ has no neighbors in $V(G'')$, and $H$ together with $G[\{x,y_1,z_1,z_2\} ]$ is an odd induced subgraph of order at least $2|V(G)|\slash 7$, which is a contradiction.

    If $d_G(x)= 3$, then  $|I|=4$, $y_1$ is not adjacent to $y_2$ and $|N_G(y_1)\cap I|=|N_G(y_1)\cap I|=2$. Let $N_G(y_1)\cap I=\{z_1,z_2\}$. Let $G^*=G- \{p,x,y_2\}-I$. Then $G^*$ contains an odd induced subgraph $H^*$ of order at least $2|V(G^*)|\slash 7 = 2(|V(G)|-7)\slash 7$ and $H^*$ together with the edge $px$ if $y_1 \notin V(H^*)$ or $H^*$ together with the edge set $\{y_1z_1,y_1z_2\}$ if $y_1\in V(H^*)$ is an odd induced subgraph of order at least $2|V(G)|\slash 7$, which is a contradiction. 

    Now we have that $d_G(x)=4$ and $3 \le |I| \le 6$.

     \noindent\textbf{Case 1:} $5\le |I|\le 6$.
     
     Without loss of generality, assume $|N_G(y_1)\cap I|=|N_G(y_2)\cap I|= 2$ and $(N_G(y_1)\cap I) \cap (N_G(y_2)\cap I)= \emptyset$. Then $y_1$ is not adjacent to $y_2$. Let $W=N_G(x)\cup N_G(y_1) \cup N_G(y_2)$ and let $G^*$ be the graph obtained from $G$ by deleting $W$ and all isolated vertices in the resulting graph. Then $G^*$ contains an odd induced subgraph $H^*$ of order at least $2|V(G^*)|\slash 7 \ge 2(|V(G)|-19)\slash 7$. Let $N_G(y_1)\cap I=\{z_1,z_2\}$ and $N_G(y_2)\cap I= \{z_3,z_4\}$. Then $H^*$ together with $G[\{p,x,y_1,y_2,z_1,z_2,z_3,z_4\}]$ is an odd induced subgraph of order at least $2|V(G)|\slash 7$, which is a contradiction. 

      \noindent\textbf{Case 2:} $3 \le |I|\le 4$.
      
    If $|N_G(y_{i_0})\cap I|= 2$ for some $i_0\in \{1,2,3\}$, then let $W=N_G(x)\cup N_G(y_{i_0})$ and let $G^*$ be the graph obtained from $G$ by deleting $W$ and all isolated vertices in the resulting graph. Then $G^*$ contains an odd induced subgraph $H^*$ of order at least $2|V(G^*)|\slash 7 \ge 2(|V(G)|-13)\slash 7$. Let $N_G(y_{i_0})\cap I=\{z_1,z_2\}$. $H^*$ together with $G[\{x,y_{i_0},z_1,z_2\}]$ is an odd induced subgraph of order at least $2|V(G)|\slash 7$, which is a contradiction. 

    Thus $|N_G(y_1)\cap I|=|N_G(y_2)\cap I|=|N_G(y_3)\cap I|= 1$ and $|I|=3$. Let $N_G(y_i)\cap I=\{z_i\}$ for $i\in \{1,2,3\}$.
    
    If each $y_i$ has no neighbors in $V(G'')$, then $V(G)=N_G[x]\cup I$ and $G[\{x,y_1,y_2,y_3\}]$ (if $y_i$ is not adjacent to $y_j$ for every $i,j\in \{1,2,3\}$), or $ G[\{p,x,y_{i_1},y_{i_2},z_{i_1},z_{i_2}\}]$ (if $y_{i_1}$ is adjacent to $y_{i_2}$ for some $i_1,i_2\in \{1,2,3\}$) is an odd induced subgraph of order at least $2|V(G)|\slash 7$, which is a contradiction. 
   
    So, without loss of generality, assume that $y_1$ has a neighbor in $V(G'')$. Let $G^*=G- \{p,x,y_2,y_3\} - I$. Then $G^*$ has no isolated vertices and it contains an odd induced subgraph $H^*$ of order at least $2|V(G^*)|\slash 7 = 2(|V(G)|-7)\slash 7$. Then $H^*$ together with the edge $px$ if $y_1\notin V(H^*)$ or $H^*$ together with the edge set $\{xy_1,y_1z_1\}$ if $y_1\in V(H^*)$ is an odd induced subgraph of order at least $2|V(G)|\slash 7$, which is a contradiction.     
\end{proof}

\begin{claim}\label{claim: delta(G)=2}
    $\delta(G)>2.$ 
\end{claim}

\begin{proof}
     For the sake of contradiction, assume that $\delta(G)=2$.

    Let $p\in V(G)$ such that $N_G(p)=\{x_1,x_2\}$. Let $W=N_G(p)\cup N_G(x_1) $. Then $|W|\le 6$. Let $G'=G-W$ and let $I$ be the set of isolated vertices in $G'$. Since $\delta(G)=2$, $|I|\le 6$. If $|W\cup I|\le 7$, then the odd induced subgraph of order at least $2(|V(G)|-7)\slash 7$ in $G'-I$, together with the edge $px_1$ is an odd induced subgraph of order at least $2|V(G)|\slash 7$, which is a contradiction. Thus, $|W\cup I|\ge 8$ and $2\le |I|\le 6$.

    \begin{subclaim}\label{claim: no tri and > 2 iv}
    Each vertex in $W$ has at most one neighbor in $I$.
    \end{subclaim}

    \begin{proof}
        
   Suppose that $u\in W$ has two neighbors in $I$.
   
   Note that $N_G(u)\setminus (W\cup I)\ne \emptyset$, since otherwise, let $\{z_1,z_2\} \subseteq N_G(u)\cap I$, and without loss of generality, assume that $u \in N_G(x_1)\setminus \{p\}$, and the odd induced subgraph of order at least $2(|V(G)|-12)\slash 7$ in $G'-I$, together with the edge set $\{x_1u,uz_1,uz_2\}$ is an odd induced subgraph of order at least $2|V(G)|\slash 7$, which is a contradiction. 
    
    Then every vertex in $W$ which has two neighbors in $I$, has one neighbor outside $W\cup I$. Then $|I|\le 4$, since $\delta(G)=2$. Without loss of generality, assume that $u \in N_G(x_1)\setminus \{p\}$ and $ N_G(u)\cap I=\{z_1,z_2\}$ and $N_G(u)\setminus (W\cup I) = \{k\}$. Then the odd induced subgraph of order at least $2(|V(G)|-14)\slash 7$ in the graph obtained from $G'-I-\{k\}$ by deleting all isolated vertices, together with the edge set $\{x_1u,uz_1,uz_2\}$ is an odd induced subgraph of order at least $2|V(G)|\slash 7$, which is a contradiction. 
     \end{proof}
    
    Subclaim \ref{claim: no tri and > 2 iv} implies that $|I|=2$ and $|W|=6$. Then each vertex in $W\setminus \{p,x_1\}$ has exactly one neighbor in $I$. Let $N_G(x_1)=\{p,y_1,y_2,y_3\}$ and $I=\{z_1,z_2\}$. Without loss of generality, let $N_G(z_1)=\{y_1,y_2\}$ and $N_G(z_2)=\{y_3,x_2\}$. If each vertex in $\{x_2,y_1,y_2,y_3\}$ has no neighbors outside $W \cup I$, then $V(G)=W \cup I$ and $G[\{p,x_1,y_1,y_2\}]$, or $G[\{p,x_1,y_1,y_3\}]$, or $G[\{p,x_1,y_2,y_3\}]$, or $G[y_1,z_1,x_2,z_2]$  is an odd induced subgraph of order at least $2|V(G)|\slash 7$, which is a contradiction. Without loss of generality, suppose that $N_G(y_1)\setminus (W \cup I)\ne \emptyset$. Let $W^*=W\setminus \{y_1\} $ and $G^*=G-W^*-I$. Then the odd induced subgraph $H^*$ of order at least $2(|V(G)|-7)\slash 7$ in $G^*$, together with the edge $px_1$ (if $y_1\notin V(H^*)$) or together with the edge set $\{x_1y_1,y_1z_1\}$ (if $y_1\in V(H^*)$), is an odd induced subgraph of order at least $2|V(G)|\slash 7$, which is a contradiction.
\end{proof}

   % By Claim \ref{claim: delta(G)=1} and Claim \ref{claim: delta(G)=2}, we may assume that $\delta(G)\geq 3$. 

   % Now we divide the rest proof into two cases depending on the existence of a triangle.
 \begin{claim}
    $G$ contains no triangles.
 \end{claim}
   \begin{proof}

   % \noindent\textbf{Case 1:} $G$ contains a triangle.
    For the sake of contradiction, assume that $G$ contains a triangle.
    Let $ab$ be an edge of this triangle and $W=N_G(a)\cup N_G(b)$. Then $|W|\le 7$. Let $G'=G- W$. Let $I$ be the set of isolated vertices in $G'$. 
    
    Note that $|W\cup I|\ge 8$, since otherwise the odd induced subgraph of order at least $2(|V(G)|-7)\slash 7$ in $G'-I$, together with the edge $ab$ is an odd induced subgraph of order at least $2|V(G)|\slash 7$. Therefore, $1\le |I|\le 4$.

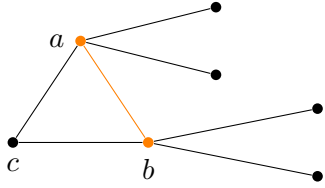
\begin{figure}[htbp]
 \centering   
\begin{tikzpicture}[node distance=1.5cm,scale=0.9]

   \draw (0,0) node[fill=orange,minimum size =4pt,circle,inner sep=1pt,label=west:$a$] (a) {};
    \draw (1,-1.5) node[fill=orange,minimum size =4pt,circle,inner sep=1pt,label=below:$b$] (b) {};
   \draw (-1,-1.5) node[fill=black,minimum size =4pt,circle,inner sep=1pt,label=below:$c$] (c) {};
    \draw (2,0.5) node[fill=black,minimum size =4pt,circle,inner sep=1pt] (y1) {};
    \draw (2,-0.5) node[fill=black,minimum size =4pt,circle,inner sep=1pt] (y2) {};
   \draw (3.5,-1) node[fill=black,minimum size =4pt,circle,inner sep=1pt] (y3) {};
    \draw (3.5,-2) node[fill=black,minimum size =4pt,circle,inner sep=1pt] (y4) {};
   
\draw (a)--(c);
\draw[orange] (a)--(b);
\draw (b)--(c);
\draw (a)--(y1);
\draw (a)--(y2);
\draw (b)--(y3);
\draw (b)--(y4);

\end{tikzpicture}

    \caption{The graph $G[W]$ when $G$ contains a triangle.}
\end{figure}

    \begin{subclaim}\label{subclaim main thm: 2 iv}
      Each vertex in $W$ has at most one neighbor in $I$.
    \end{subclaim}
    \begin{proof}
     Suppose that $u\in W$ has two neighbors in $I$.
     
    Note that $N_G(u)\setminus (W\cup I)\ne \emptyset$, since otherwise, let $\{z_1,z_2\} \subseteq N_G(u)\cap I$, and without loss of generality, assume that  $u \in N_G(a)\setminus \{b\}$ and the odd induced subgraph of order at least $2(|V(G)|-11)\slash 7$ in $G'-I$, together with the edge set $\{au,uz_1,uz_2\}$  is an odd induced subgraph of order at least $2|V(G)|\slash 7$, which is a contradiction.

  So, every vertex in $W$ which has two neighbors in $I$, has one neighbor outside $W\cup I$. Then $|I|\le 3$ since $\delta(G)\geq 3$. Without loss of generality, assume that $u \in N_G(a)\setminus \{b\}$ and $ N_G(u)\cap I=\{z_1,z_2\}$ and $N_G(u)\setminus (W\cup I) = \{k\}$. Then the odd induced subgraph of order at least $2(|V(G)|-14)\slash 7$ in the graph obtained from $G'-I-\{k\}$ by deleting all isolated vertices, together with the edge set $\{au,uz_1,uz_2\}$ is an odd induced subgraph of order at least $2|V(G)|\slash 7$. 
    \end{proof}

    Subclaim \ref{subclaim main thm: 2 iv} implies that $|I|= 1$ and $|W|=7$. Let $I=\{z\}$, $N_G(a)=\{b,c,y_1,y_2\}$ and $N_G(b)=\{a,c,y_3,y_4\}$.  Let $W'=\{w\in W: N_G(w)\setminus (W\cup \{z\})\ne \emptyset \}$.  

    \noindent \textbf{Case 1:} $W'=\emptyset$.
    
     In this case, we have that $V(G)=W\cup \{z\}$. Note that $y_1$ is adjacent to $y_2$, and $y_3$ is adjacent to $y_4$, since otherwise $G[\{a,b,y_1,y_2\}]$ or $G[\{a,b,y_3,y_4\}]$ is an odd induced subgraph of order at least $2|V(G)|\slash 7$. Moreover, if neither $y_1$ nor $y_2$ is adjacent to $c$, then the edge set $\{bc, y_{1}y_{2}\}$ forms an odd induced subgraph of order at least $2|V(G)|\slash 7$. Thus, without loss of generality, we can assume that $y_1$ is adjacent to $c$ and by symmetry, assume that $y_3$ is adjacent to $c$. Then $d_G(c)=4$ and $N_G(c)=\{a,b,y_1,y_3\}$.  Since $\delta(G)\geq3 $, $y_2$ or $y_4$ is adjacent to $z$. Then $G[\{b,c,y_2,z\}]$ or $G[\{a,c,y_4,z\}]$ is an odd induced subgraph of order at least $2|V(G)|\slash 7$, which is a contradiction. 

    \noindent \textbf{Case 2:} There is a neighbor $u$ of $z$ such that $u\in W'$.

     Let $W^*=(W\cup \{z\})\setminus \{u\}$ and $G^*=G\setminus W^*$. Then $G^*$ has no isolated vertices, and it has an odd induced subgraph $H$ of order at least $2(|V(G)|-7)\slash 7$. If $H$ contains $u$, then $G[V(H)\cup \{z,a\}]$ (if $u\in \{c,y_1,y_2\}$) or $G[V(H)\cup \{z,b\}]$ (if $u\in \{c,y_3,y_4\}$) is an odd induced subgraph of order at least $2|V(G)|\slash 7$, which is a contradiction. If not, then  $H$ together with the edge $ab$ is an odd induced subgraph of order at least $2|V(G)|\slash 7$, which is also a contradiction.
    
    \noindent \textbf{Case 3:} $W'\ne \emptyset$ and $N_G(z)\cap W'=\emptyset$.

    Note that there are some $w\in W'$ and some $u\in N_G(z)$ such that $w$ is not adjacent to $u$ since $\delta(G)\ge 3$ and $\Delta(G)\le 4$. Let $W^*=(W\cup \{z\})\setminus \{w\}$ and $G^*=G\setminus W^*$. Then $G^*$ has no isolated vertices and it has an odd induced subgraph $H$ of order at least $2(|V(G)|-7)\slash 7$. $H$ together with the edge $zu$ is an odd induced subgraph of order at least $2|V(G)|\slash 7$, which is a contradiction. 
\end{proof}   
    
    \begin{claim}\label{4-regular}
        $G$ is $4$-regular.
    \end{claim}
\begin{proof}

  For the sake of contradiction, assume that there is a vertex, say $p$, such that $N_G(p)=\{x_1,x_2,x_3\}$. Then let $N_G(x_1)=\{p,y_1,\dots,y_{d_G(x_1)-1}\}$, and $W=N_G(p)\cup N_G(x_1)$ and $G'=G-W$. Let $I$ be the set of isolated vertices in $G'$.

      \begin{figure}[htbp]
 \centering   
\begin{tikzpicture}[node distance=1.5cm,scale=0.9]

   \draw (0,0) node[fill=orange,minimum size =4pt,circle,inner sep=1pt,label=west:$p$] (p) {};
    \draw (1,1.5) node[fill=orange,minimum size =4pt,circle,inner sep=1pt,label=below:$x_1$] (x1) {};
   \draw (1,0) node[fill=black,minimum size =4pt,circle,inner sep=1pt,label=below:$x_2$] (x2) {};
    \draw (1,-1.5) node[fill=black,minimum size =4pt,circle,inner sep=1pt,label=below:$x_3$] (x3) {};
   \draw (3,2) node[fill=black,minimum size =4pt,circle,inner sep=1pt] (z11) {};
    \draw (3,1.5) node[fill=black,minimum size =4pt,circle,inner sep=1pt] (z12) {};
    \draw (3,1) node[fill=black,minimum size =4pt,circle,inner sep=1pt] (z13) {};

\draw[orange] (p)--(x1);
\draw (p)--(x2);
\draw (p)--(x3);

\draw (x1)--(z11);
\draw (x1)--(z12);
\draw (x1)--(z13);

\end{tikzpicture}

    \caption{The graph $G[W]$ when $G$ contains no triangles and $\delta(G)=3$.}
    \end{figure}
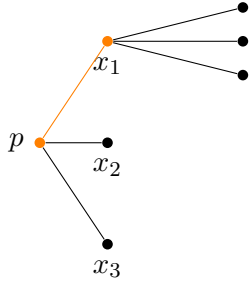
   
     Note that $|W\cup I|\ge 8$, since otherwise the odd induced subgraph of order at least $2(|V(G)|-7)\slash 7$ in $G'-I$, together with the edge $px_1$ is an odd induced subgraph of order at least $2|V(G)|\slash 7$, which is a contradiction. Therefore, $1\le |I|\le 5$.

    If $|I|=4$ or 5, then  $|N_G(y_{i_0})\cap I|=3$ for some $i_0\in [d_G(x_1)-1]$ or $|N_G(x_{j_0})\cap I|=3$ for some $j_0\in \{2,3\}$. Without loss of generality, let $N_G(y_1)\cap I=\{z_1,z_2,z_3\}$. Then the odd induced subgraph of order at least $2(|V(G)|-12)\slash 7$ in $G'-I$, together with the edge set $\{y_1z_1,y_1z_2,y_1z_3\}$ is an odd induced subgraph of order at least $2|V(G)|\slash 7$.

     If $|I|=2$ or 3, then  $|N_G(y_{i_0})\cap I|\ge 2$ for some $i_0\in [d_G(x_1)-1]$ or $|N_G(x_{j_0})\cap I|\ge 2$ for some $j_0\in \{2,3\}$. Without loss of generality, let $\{z_1,z_2\}\subseteq N_G(y_1)\cap I$. Let $G^*=G-W-I-N_G(y_1)$. Then the odd induced subgraph of order at least $2(|V(G)|-14)\slash 7$  in the graph obtained from $G^*$ by deleting all isolated vertices, together with the edge set $\{x_1y_1,y_1z_1,y_1z_2\}$ is an odd induced subgraph of order at least $2|V(G)|\slash 7$.

     If $|I|=1$, then $|W|= 7$. Let $I=\{z_1\}$ and $W'=\{w\in W ~|~ N_G(w)\setminus (W\cup \{z_1\})\ne \emptyset \}$. Note that  $W'\ne \emptyset$, since otherwise $V(G)=W\cup\{z_1\}$ and $G[\{p,x_1,x_2,x_3\}]$ is an odd induced subgraph of order at least $2|V(G)|\slash 7$. If there is a neighbor of $z_1$, without loss of generality, say $y_1$, such that $y_1\in W'$, then let $W^*=(W\cup \{z_1\})\setminus \{y_1\}$ and $G^*=G- W^*$. The odd induced subgraph $H^*$ of order at least $2(|V(G)|-7)\slash 7$ in $G^*$, together with the edge $px_1$ (if $y_1\notin V(H^*)$) or together with the edge set $\{x_1y_1,y_1z_1\}$ (if $y_1\in V(H^*)$), is an odd induced subgraph of order at least $2|V(G)|\slash 7$, which is a contradiction. Therefore,  $N_G(z_1)\cap W'=\emptyset$. Since $d_G(z_1)\ge 3$, we can always find  $w \in W'$ and $u\in N_G(z_1)$ such that $w$ is not adjacent to $u$. Let $W^*=(W\cup \{z_1\})\setminus \{w\}$ and $G^*=G- W^*$. Then the odd induced subgraph $H^*$ of order at least $2(|V(G)|-7)\slash 7$ in $G^*$, together with the edge $z_1u$ is an odd induced subgraph of order at least $2|V(G)|\slash 7$. 
  \end{proof}

    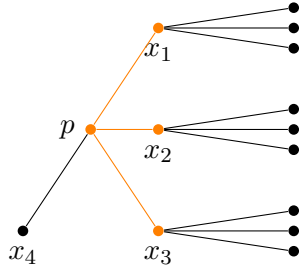
\begin{figure}[htbp]
 \centering   
\begin{tikzpicture}[node distance=1.5cm,scale=0.9]

   \draw (0,0) node[fill=orange,minimum size =4pt,circle,inner sep=1pt,label=west:$p$] (p) {};
    \draw (1,1.5) node[fill=orange,minimum size =4pt,circle,inner sep=1pt,label=below:$x_1$] (x1) {};
   \draw (1,0) node[fill=orange,minimum size =4pt,circle,inner sep=1pt,label=below:$x_2$] (x2) {};
    \draw (1,-1.5) node[fill=orange,minimum size =4pt,circle,inner sep=1pt,label=below:$x_3$] (x3) {};
    \draw (-1,-1.5) node[fill=black,minimum size =4pt,circle,inner sep=1pt,label=below:$x_4$] (x4) {};
   \draw (3,1.8) node[fill=black,minimum size =4pt,circle,inner sep=1pt] (z11) {};
    \draw (3,1.5) node[fill=black,minimum size =4pt,circle,inner sep=1pt] (z12) {};
    \draw (3,1.2) node[fill=black,minimum size =4pt,circle,inner sep=1pt] (z13) {};
    \draw (3,0.3) node[fill=black,minimum size =4pt,circle,inner sep=1pt] (z21) {};
    \draw (3,0) node[fill=black,minimum size =4pt,circle,inner sep=1pt] (z22) {};
    \draw (3,-0.3) node[fill=black,minimum size =4pt,circle,inner sep=1pt] (z23) {};
    \draw (3,-1.2) node[fill=black,minimum size =4pt,circle,inner sep=1pt] (z31) {};
    \draw (3,-1.5) node[fill=black,minimum size =4pt,circle,inner sep=1pt] (z32) {};
    \draw (3,-1.8) node[fill=black,minimum size =4pt,circle,inner sep=1pt] (z33) {};
   
\draw[orange] (p)--(x1);
\draw[orange] (p)--(x2);
\draw[orange] (p)--(x3);
\draw (p)--(x4);
\draw (x1)--(z11);
\draw (x1)--(z12);
\draw (x1)--(z13);
\draw (x2)--(z21);
\draw (x2)--(z22);
\draw (x2)--(z23);
\draw (x3)--(z31);
\draw (x3)--(z32);
\draw (x3)--(z33);

\end{tikzpicture}

    \caption{The graph $G[W]$ when $G$ contains no triangles and is 4-regular.}
    \end{figure}

   Let $p$ be an arbitrary vertex of $G$. Let $N_G(p)=\{x_1,x_2,x_3,x_4\}$, $W=N_G(p)\cup N_G(x_1) \cup N_G(x_2) \cup N_G(x_3)$ and $G'=G- W$. Then $|W|\le 14$. Let $I$ be the set of isolated vertices in $G'$. 
   
   Note that $|W\cup I|\ge 15$, since otherwise $G'-I$ contains an odd induced subgraph $H'$ of order at least $2(|V(G)|-14)\slash 7$, and $H'$ together with the edge set $\{px_1,px_2,px_3\}$ is an odd induced subgraph of order at least $2|V(G)|\slash 7$. Therefore, $1\le |I|\le 7$.

Firstly, we give three useful observations to describe the neighbors of the vertex in $W$.

    \begin{observation}\label{obv: 4-reg and 3 neig in I}
        Each vertex in $W$ has at most two neighbors in $I$.
    \end{observation}
 \begin{proof}
     Suppose that 
 there is a vertex in the set $W$ that has three neighbors in $I$.
       Since the neighborhoods of $p$, $x_1$, $x_2$ and $x_3$ are all contained in $W$, without loss of generality, we may assume that $x_4$, or one vertex in $N_G(x_1)\setminus \{p\}$, say $y_1$, has three neighbors in $I$.

       Firstly, if $x_4$ has three neighbors in $I$, say $z_1,z_2$ and $z_3$, then the odd induced subgraph of order at least $2(|V(G)|-|W|-|I|)\slash 7 \ge 2(|V(G)|-21)\slash 7$ in $G'-I$, together with the edge set $\{z_1x_4,z_2x_4,x_4p,px_1,px_2\}$, is an odd induced subgraph of order at least $2|V(G)|\slash 7$, which is a contradiction.

       If $y_1$ has three neighbors in $I$, also say $z_1,z_2$ and $z_3$, then the odd induced subgraph of order at least $2(|V(G)|-|W|-|I|)\slash 7 \ge 2(|V(G)|-21)\slash 7$ in $G'-I$, together with the edge set $\{px_2,y_1z_1,y_1z_2,y_1z_3\}$, is an odd induced subgraph of order at least $2|V(G)|\slash 7$, which is a contradiction.
\end{proof}

   \begin{observation}\label{obv: 4-reg and 2 I of x4}
        If $|I|\le 5$,  then $x_4$ has at most one neighbor in $I$.
     \end{observation}

 \begin{proof}
    For the sake of contradiction, assume that $N_G(x_4)\cap I=\{z_1,z_2\}$.
    
    Firstly, if $N_G(x_4)\subseteq W\cup I$, then the odd induced subgraph of order at least $2(|V(G)|-|W|-|I|)\slash 7 \ge 2(|V(G)|-19)\slash 7$ in $G'$, together with the edge set $\{z_1x_4,z_2x_4,x_4p,px_1,px_2\}$, is an odd induced subgraph of order at least $2|V(G)|\slash 7$.  

    So, let $N_G(x_4)\setminus (W\cup I)=\{k\}$. Let $G^*=G-W-I-\{k\}$. If $G^*$ contains at most one isolated vertex, the odd induced subgraph of order at least $$2(|V(G)|-|W|-|I|-1-1)\slash 7 \ge 2(|V(G)|-21)\slash 7$$ in $G^*$, together with the edge set $\{z_1x_4,z_2x_4,x_4p,px_1,px_2\}$, is an odd induced subgraph of order at least $2|V(G)|\slash 7$.
      Therefore, $G^*$ contains at least two isolated vertices, say $t_1$ and $t_2$. Then let $G_2^*=G^*- N_G(k)$. Then the odd induced subgraph of order at least $$2(|V(G)|-|W|-|I|-1-6)\slash 7 \ge 2(|V(G)|-26)\slash 7$$ in the graph obtained from $G_2^*$ by deleting all isolated vertices, together with the edge set $\{px_1,px_2, px_4, x_4z_1, x_4k, kt_1, kt_2\}$, is an odd induced subgraph of order at least $2|V(G)|\slash 7$.
           \end{proof} 

    \begin{observation}\label{obv: 4-regular x4 1 in I and 1 in G'}
        If $|W\cup I|=15$, then $x_4$ has no neighbors in $I$ or has no neighbors outside $W\cup I$.        
    \end{observation}
\begin{proof}

    For the sake of contradiction, assume that $z\in N_G(x_4)\cap I$ and $N_G(x_4)\setminus (W\cup I)\ne \emptyset$.  Let $W^*=W\setminus \{x_4\}$ and $G^*=G-W^*-I$. Then the odd induced subgraph $H^*$ of order at least $2(|V(G)|-14)\slash 7$ in $G^*$, together with the edge set $\{px_1,px_2,px_3\}$ (if $x_4\notin V(H^*)$), or together with the edge set $\{x_4z,x_4p,px_1,px_2\}$ (if $x_4\in V(H^*)$), is an odd induced subgraph of order at least $2|V(G)|\slash 7$.
\end{proof}

    Then $|I|\le 4$, and moreover $12\le |W|\le 14$, and if $|W|=12$ then $|I|= 3$, by Observation~\ref{obv: 4-reg and 3 neig in I} and Observation~\ref{obv: 4-reg and 2 I of x4}. 
    
    Let $W_0=\{w\in W : |N_G(w)\cap I|=0\}$, $W_1=\{w\in W : |N_G(w)\cap I|=1\}$ and $W_2=\{w\in W: |N_G(w)\cap I|=2\}$. Observation~\ref{obv: 4-reg and 3 neig in I} implies that $W=W_0\cup W_1 \cup W_2$. Observation \ref{obv: 4-reg and 2 I of x4} implies that $x_4\notin W_2$.
    
    If $|W|=12$ and $|I|=3$, then $4\le |W_2|\le 6$. If $|W_2|=4$,  then $|W_1|=4$, $|W_0|=0$, and by Observation \ref{obv: 4-reg and 2 I of x4} and \ref{obv: 4-regular x4 1 in I and 1 in G'}, $|N_G(x_4) \cap I|=1$ and $N_G(x_4) \setminus (W\cup I)=\emptyset$. Without loss of generality, assume that $y_1\in N_G(x_1)\setminus (N_G(x_2)\cup N_G(x_3))$ such that $y_1\in W_2$. Let $z\in (N_G(y_1)\cap I)\setminus N_G(x_4)$. The odd induced subgraph of order at least $2(|V(G)|-12-3-4)\slash 7 \ge 2(|V(G)|-19)\slash 7$ in the graph obtained from $G'-I-N_G(y_1)$ by deleting all isolated vertices, together with the edge set $\{px_2,px_3, px_4, y_1z\}$ (if $x_4 \notin N_G(y_1)$) or together with $G[\{p,x_4,y_1\}\cup I]$ (if $x_4 \in N_G(y_1)$), is an odd induced subgraph of order at least $2|V(G)|\slash 7$. If $|W_2|\ge 5$, then $|W_1|=2$ or $|W_2|= 6$. Therefore, without loss of generality, assume that $y_1,y_2\in N_G(x_1)\setminus (N_G(x_2)\cup N_G(x_3))$ such that $y_1\in W_2$ and $y_2\in W_1\cup W_2$. Let $K=N_G(y_1)\cup N_G(y_2)\setminus (W\cup I)$. Then $|K|\le 3$. Let $G^*=G-W-I-K$ and let $I^*$ be the set of isolated vertices in $G^*$. Then we have that 
    
    \begin{align*}
    |I^*| \le \left \lfloor \frac{1}{4}\left ( \sum_{v\in W\setminus \{p,x_1,x_2,x_3\}} |N_G(v)\cap V(G')|+ \sum_{u\in K}(d_G(u)-1) - 4|I|- |K| \right) \right \rfloor. 
    \end{align*}

    \noindent If $|K|\le 2$, we have that $$|I^*|\le \max \{\left \lfloor \frac{1}{4} (22+ 6-12-2) \right \rfloor, \left \lfloor \frac{1}{4} (22+ 3-12-1) \right \rfloor , \left \lfloor \frac{1}{4} (22 -12) \right \rfloor \} \\=3,$$ and the odd induced subgraph of order at least $2(|V(G^*)|-|I^*|)\slash 7 \ge 2(|V(G)|-20)\slash 7$ in $G^*-I^*$, together with the edge set $\{px_1,px_2, px_3, x_1y_1, x_1y_2\}$, is an odd induced subgraph of order at least $2|V(G)|\slash 7$. Thus $|K|=3$ and $y_2\in W_1$, and $|I^*|\le \left \lfloor \frac{1}{4} (22+ 9-12-3) \right \rfloor =4$. If $|I^*|\le 3$, similarly, the odd induced subgraph of order at least $2(|V(G^*)|-|I^*|)\slash 7 \ge 2(|V(G)|-21)\slash 7$ in $G^*-I^*$, together with the edge set $\{px_1,px_2, px_3, x_1y_1, x_1y_2\}$, is an odd induced subgraph of order at least $2|V(G)|\slash 7$. If $|I^*|=4$, one can find that $V(G)=W\cup I
    \cup K\cup I^*$ and $x_4$ has at least two neighbors in $I^*$, say $t_1$ and $t_2$. Then $G[\{p,x_1,x_2,x_4,y_1,y_2,t_1,t_2\}]$ is an odd induced subgraph of order $8 >  2|V(G)|\slash 7$.

%\pgfplotsset{compat=1.18}

\begin{figure}[htbp]
\centering

\begin{subfigure}{0.45\textwidth}
\centering
\begin{tikzpicture}
[
    scale=1.4,
    vertex/.style={circle, fill=black, inner sep=1pt, minimum size=4pt, draw=none},
    every label/.style={font=\small, inner sep=1pt}
]

% 顶点
\node[vertex, label=above:$p$] (p) at (0,0) {};
\node[vertex, label=above:$x_4$] (x4) at (1,0.8) {};
\node[vertex, label=left:$x_1$] (x1) at (-1,-0.8) {};
\node[vertex, label=left:$x_2$] (x2) at (0,-0.8) {};
\node[vertex, label=right:$x_3$] (x3) at (1,-0.8) {};

\node[vertex] (y3) at (-0.6,-1.5) {};
\node[vertex] (y4) at (0.6,-1.5) {};

% 红色实线部分
\draw[orange, thick] (p)--(x1);
\draw[orange, thick] (p)--(x2);
\draw[orange, thick] (p)--(x3);

% 黑色实线部分
\draw (p)--(x4);
\draw (x1)--(y3);
\draw (x2)--(y3);
\draw (x2)--(y4);
\draw (x3)--(y4);

% % 虚线外接部分
% \draw[dashed] (x1) -- ++(-0.8,-0.8);
% \draw[dashed] (x1) -- ++(-1.0,0.2);
% \draw[dashed] (x3) -- ++(0.8,-0.8);
% \draw[dashed] (x3) -- ++(1.0,0.2);
\end{tikzpicture}
\caption{One common edge}
\end{subfigure}
\hfill
\begin{subfigure}{0.45\textwidth}
\centering
\begin{tikzpicture}
[
    scale=1.4,
    vertex/.style={circle, fill=black, inner sep=1pt, minimum size=4pt, draw=none},
    every label/.style={font=\small, inner sep=1pt}
]

% 顶点
\node[vertex, label=above:$p$] (p) at (0,0) {};
\node[vertex, label=above:$x_4$] (x4) at (1,0.8) {};
\node[vertex, label=left:$x_1$] (x1) at (-1,-0.8) {};
\node[vertex, label=left:$x_2$] (x2) at (0,-0.8) {};
\node[vertex, label=right:$x_3$] (x3) at (1,-0.8) {};

\node[vertex] (y3) at (0,-1.5) {};

% 红色实线部分
\draw[orange, thick] (p)--(x1);
\draw[orange, thick] (p)--(x2);
\draw[orange, thick] (p)--(x3);

% 黑色实线部分
\draw (p)--(x4);
\draw (x1)--(y3);
\draw (x2)--(y3);
\draw (x3)--(y3);

% % 虚线外接部分
% \draw[dashed] (x1) -- ++(-0.8,-0.8);
% \draw[dashed] (x1) -- ++(-1.0,0.2);
% \draw[dashed] (x3) -- ++(0.8,-0.8);
% \draw[dashed] (x3) -- ++(1.0,0.2);
\end{tikzpicture}
\caption{Two common edges}
\end{subfigure}

    \caption{Two different cycles of length four with common edges.}

\end{figure}
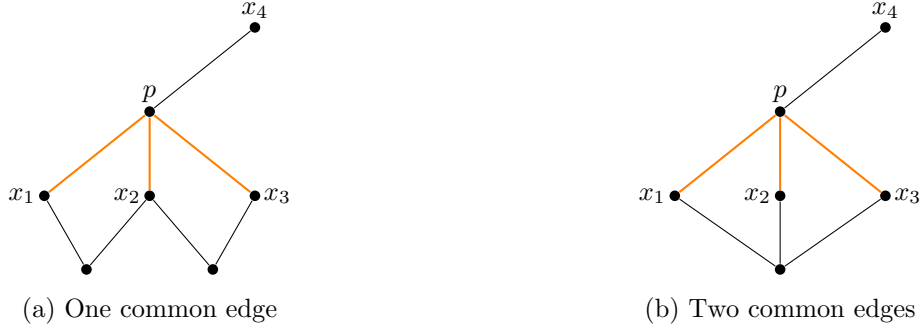

    \begin{claim}\label{claim: no two C4}
        $G$ does not contain two different cycles of length four that have common edges.
    \end{claim}
 
    \begin{proof}
        
 If $G$ contains two different cycles of length four that have common edges, we can let $p$ be the vertex in the two cycles of length four which is adjacent to three vertices in the two cycles, and let these three vertices be $x_1,x_2$ and $x_3$, and let the other neighbor of $p$ in $G$ be $x_4$. The claim follows by doing the same discussions all above. 
 \end{proof}

    If $|W|=13$, we have that $|I|= 2$, $3$ or $4$.

        If $|W|=13$ and $|I|=2$, let $N_G(x_1)=\{p,y_1,y_2,y_3\}$, $N_G(x_2)=\{p,y_3,y_4,y_5\}$, $N_G(x_3)=\{p,y_6,y_7,y_8\}$ and $I=\{z_1,z_2\}$. Note that $N_G(x_4) \setminus (W \cup I)\ne \emptyset$ by Claim \ref{claim: no two C4}. Then by Observation \ref{obv: 4-regular x4 1 in I and 1 in G'}, $x_4\in W_0$. If $y_3\in W_1 \cup W_2$, let $z_1\in N_G(y_3)\cap I$, and by Claim \ref{claim: no two C4}, $N_G(z_1)\subseteq \{y_3,y_6,y_7,y_8\}$ but $\{y_6,y_7,y_8\}\not\subseteq N_G(z_1)$, which contradicts to $d_G(z_1)=4$. Thus $y_3\in W_0$. Moreover, by Claim \ref{claim: no two C4}, every two $y_i's$ which are adjacent to the same $x_j$, cannot belong to $W_2$ at the same time, and $1\le |W_2|\le 3$. Without loss of generality, assume that $y_1\in W_2$ and $y_2\in W_1$ and $N_G(y_2)\cap I=\{z_1\}$. Let $K=N_G(y_1)\cup N_G(y_2)\setminus (W\cup I)$. Then $|K|\le 3$. Let $G^*=G-W-I-K$ and let $I^*$ be the set of isolated vertices in $G^*$. Similarly, if $|K|\le 2$, we have that $$|I^*|\le  \max \{\left \lfloor \frac{1}{4} (26+ 6-8-2) \right \rfloor, \left \lfloor \frac{1}{4} (26+ 3-8-1) \right \rfloor , \left \lfloor \frac{1}{4} (26 -8) \right \rfloor \}  =5.$$  If $|I^*|\le 4$, then the odd induced subgraph of order at least $2(|V(G^*)|-|I^*|)\slash 7 \ge 2(|V(G)|-21)\slash 7$ in $G^*-I^*$, together with the edge set $\{px_1,px_2, px_3, x_1y_1, x_1y_2\}$, is an odd induced subgraph of order at least $2|V(G)|\slash 7$. If $|I^*|=5$, then $|K|=2$ and one vertex in $K$, say $k$, has three neighbors in $I^*$, say $t_1,t_2$ and $t_3$. Then the odd induced subgraph of order at least $2(|V(G^*)|-|I^*|)\slash 7 \ge 2(|V(G)|-22)\slash 7$ in $G^*-I^*$, together with the edge set $\{px_1,px_2, px_3, kt_1, kt_2,kt_3\}$, is an odd induced subgraph of order at least $2|V(G)|\slash 7$.

     Therefore, $|K|=3$ and $|I^*|\le \left \lfloor \frac{1}{4} (26+ 9-8-3) \right \rfloor =6$. If $|I^*|\le 3$, the odd induced subgraph of order at least $2(|V(G^*)|-|I^*|)\slash 7 \ge 2(|V(G)|-21)\slash 7$ in $G^*-I^*$, together with the edge set $\{px_1,px_2, px_3, x_1y_1, x_1y_2\}$, is an odd induced subgraph of order at least $2|V(G)|\slash 7$. If $4\le |I^*|\le 6$, then some vertex in $K$, say $k$, has two neighbors in $I^*$, say $t_1$ and $t_2$. Let $G_1=G-W-I-K-I^*- N_G(k)$. Then the graph obtained from $G_1$ by deleting all isolated vertices, contains an odd induced subgraph $H_1$ of order at least $2(|V(G)|-13-2-3-6-1-3)\slash 7 \ge 2(|V(G)|-28)\slash 7$. If $k$ is adjacent to $y_1$, then $k$ is not adjacent to $y_2$ and $H_1$ together with the edge set $\{px_3, y_1z_1, y_1z_2, y_1k, kt_1, kt_2\}$ is an odd induced subgraph of order at least $2|V(G)|\slash 7$.  If $k$ is adjacent to $y_2$, then $k$ is not adjacent to $y_1$ and $H_1$ together with the edge set $\{px_3, y_1z_2, y_2k, kt_1, kt_2\}$ is an odd induced subgraph of order at least $2|V(G)|\slash 7$.

        If $|W|=13$ and $|I|=3$, let $N_G(x_1)=\{p,y_1,y_2,y_3\}$, $N_G(x_2)=\{p,y_3,y_4,y_5\}$ and $N_G(x_3)=\{p,y_6,y_7,y_8\}$. Note that $|W_2|\ge 3$, and by Claim \ref{claim: no two C4}, if $y_3\in W_2$, we have $\{y_1,y_2,y_4,y_5\}\cap W_2=\emptyset$. Without loss of generality, assume that $y_1\in W_2$ and $y_2\in W_1\cup W_2$. 
  
      If $y_2\in W_2$, then $|(N_G(y_1)\cap I)\cap  (N_G(y_2)\cap I)|=1$ by Claim \ref{claim: no two C4}. Let $K=N_G(y_1)\cup N_G(y_2)\setminus (W\cup I)$. Then $|K|\le 2$. Let $G^*=G-W-I-K$ and let $I^*$ be the set of isolated vertices in $G^*$. Similarly, we have that 
     $$|I^*|\le  \max \{\left \lfloor \frac{1}{4} (26+ 6-12-2) \right \rfloor, \left \lfloor \frac{1}{4} (26+ 3-12-1) \right \rfloor , \left \lfloor \frac{1}{4} (26 -12) \right \rfloor \}  =4.$$ If $|I^*|\le 3$, then the odd induced subgraph of order at least $2(|V(G^*)|-|I^*|)\slash 7 \ge 2(|V(G)|-21)\slash 7$ in $G^*-I^*$, together with the edge set $\{px_1,px_2, px_3, x_1y_1, x_1y_2\}$, is an odd induced subgraph of order at least $2|V(G)|\slash 7$. Therefore, $|I^*|=4$. If one vertex in $K$, say $k$, has three neighbors in $I^*$, say $t_1,t_2$ and $t_3$, then the odd induced subgraph of order at least $2(|V(G^*)|-|I^*|)\slash 7 \ge 2(|V(G)|-22)\slash 7$ in $G^*-I^*$, together with the edge set $\{px_1,px_2, px_3, kt_1, kt_2,kt_3\}$, is an odd induced subgraph of order at least $2|V(G)|\slash 7$. If every vertex in $K$ has at most two neighbors in $I^*$, then $|K|=2$. Note that 
     \begin{align*}
         \sum_{v\in W\setminus \{p,x_1,x_2,x_3\}}|N_G(v)\cap V(G')| &\ge \sum_{t\in I^*}|N_G(t)\cap W|+ {\textstyle \sum_{z\in I}|N_G(z)\cap W|} +|K| \\
         &\ge  16-4+12 +2\\
         &= 26 \\
         &\ge \sum_{v\in W\setminus \{p,x_1,x_2,x_3\}}|N_G(v)\cap V(G')|. 
     \end{align*} 
      Then $\bigcup_{v\in W\setminus \{p,x_1,x_2,x_3\}}N_G(v)\subseteq \{p,x_1,x_2,x_3\}\cup I \cup K \cup I^*$, which means that $y_7,y_8$ have no neighbors in $G^*-I^*$ and $y_1,y_2$ are not adjacent to $y_7,y_8$. Now the odd induced subgraph of order at least $2(|V(G^*)|-|I^*|)\slash 7 \ge 2(|V(G)|-22)\slash 7$ in $G^*-I^*$, together with the edge set $\{px_1,px_2, px_3, x_1y_1, x_1y_2,x_3y_7,x_3y_8\}$, is an odd induced subgraph of order at least $2|V(G)|\slash 7$.
     
     Therefore, $y_2\in W_1$, and we may assume that $\{y_4,y_5\}\not\subset W_2$, $\{y_6,y_7\}\not\subset W_2$, $\{y_7,y_8\}\not\subset W_2$ and $\{y_6,y_8\}\not\subset W_2$ otherwise we can do the same discussion above. Recall that $|W_2|\ge 3$ and if $y_3\in W_2$, then $\{y_1,y_2,y_4,y_5\}\cap W_2=\emptyset$. So $|W_2|=3$ and $|W_1|=6$. Without loss of generality, let $y_6\in W_2$ and $y_7,y_8\in W_1$. Let $N_G(y_6)\cap I=\{z_1,z_2\}$. By Claim \ref{claim: no two C4}, $N_G(y_7)\cap \{z_1,z_2\}=\emptyset$ or $N_G(y_8)\cap \{z_1,z_2\}=\emptyset$. Without loss of generality, let $N_G(y_7)\cap \{z_1,z_2\}=\emptyset$. Then $(N_G(y_6)\cap I)\cup (N_G(y_7)\cap I)=I$. Let $G_2^*=G-W-N_G(y_6)-N_G(y_7)$. Then the graph obtained from $G_2^*$ by deleting all isolated vertices, contains an odd induced subgraph $H_2^*$ of order at least $2(|V(G)|-13-3-3-9)\slash 7 \ge 2(|V(G)|-28)\slash 7$. $H_2^*$ together with the edge set $\{p x_1, p x_2,p x_3, x_3 y_6, x_3 y_7, y_6z_1, y_6z_2 \}$ is an odd induced subgraph of order at least $2|V(G)|\slash 7$.

        If $|W|=13$ and $|I|=4$, then $|W_2|\ge 7$. But by Claim \ref{claim: no two C4}, we have that $|W_2|\le 6$, which is a contradiction.

\begin{figure}[htbp]
 \centering

\begin{tikzpicture}[
    scale=1.4,
    vertex/.style={circle, fill=black, inner sep=1pt, minimum size=4pt, draw=none},
    every label/.style={font=\small, inner sep=1pt}
]

% 顶点
\node[vertex, label=left:$p$] (v) at (0,0) {};
\node[vertex, label=left:$x_3$] (x3) at (0.8,0.8) {};
\node[vertex, label=left:$x_4$] (x4) at (0,1) {};
\node[vertex, label=left:$x_1$] (x1) at (-0.6,-1) {};
\node[vertex, label=left:$x_2$] (x2) at (0.6,-1) {};
\node[vertex, label=left:$y_3$] (y3) at (0,-2) {};

% 红色实线部分
\draw[orange, thick] (v)--(x1);
\draw[orange, thick] (v)--(x2);
\draw[orange, thick] (v)--(x3);

% 黑色实线部分
\draw (v)--(x4);
\draw (x1)--(y3);
\draw (x2)--(y3);

% % 虚线外接部分
% \draw[dashed] (x1) -- ++(-0.8,-0.8);
% \draw[dashed] (x1) -- ++(-1.0,0.4);
% \draw[dashed] (x2) -- ++(0.8,-0.8);
% \draw[dashed] (x2) -- ++(1.0,0.4);
% \draw[dashed] (x3) -- ++(0.8,0.6);
% \draw[dashed] (x3) -- ++(0.8,0);
% \draw[dashed] (x3) -- ++(0.8,-0.6);

\end{tikzpicture}

\caption{ A cycle of length four.}
    \end{figure}
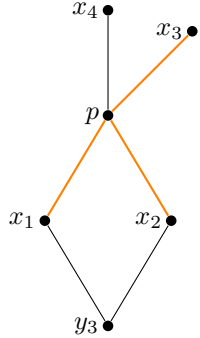
    \begin{claim}\label{claim: no C4}
        $G$ contains no cycles of length four.
    \end{claim}
 \begin{proof}

    If $G$ contains a cycle of length four, we can let $p$ be the vertex in this cycle which is adjacent to two vertices in this cycle, and let these two vertices be $x_1,x_2$, and let the other neighbors of $p$ in $G$ be $x_3,x_4$.  The claim follows by doing the same discussions all above. 
 \end{proof}

    By Observation \ref{obv: 4-reg and 2 I of x4}, all the above discussions and Claim \ref{claim: no C4}, we have $|W|=14$ and $|I|=1$. Moreover, let $I=\{z\}$. Then $z$ is adjacent to $x_4$ and Observation \ref{obv: 4-regular x4 1 in I and 1 in G'} implies that  $N_G(x_4)\subseteq W\cup I $, but at this moment $G$ will contain a cycle of length four since $d_G(x_4)=4$, which is a contradiction.

  \begin{figure}[htbp]
 \centering   
\begin{tikzpicture}[scale=1.05,
  every node/.style={circle, fill=black, inner sep=1.5pt},
  edge/.style={line width=0.6pt}]
  % --- parameters ---
  \def\n{7}              % number of vertices
  \def\r{2.2cm}          % radius of the circle

  % --- place vertices on a circle and label radially ---
  \foreach \i in {1,...,\n} {
    \pgfmathsetmacro\ang{90 - 360*(\i-1)/\n}
    \node[label={[label distance=1.6mm]\ang:$v_{\i}$}] (v\i) at (\ang:\r) {};
  }

  % --- draw all edges of K_7 except a Hamilton cycle v1-v2-...-v7-v1 ---
  \foreach \i in {1,...,6}{
    \pgfmathtruncatemacro\nexti{mod(\i,\n)+1} % = i+1
    \foreach \j in {\the\numexpr\i+1\relax,...,7}{
      % skip edges on the Hamilton cycle: (i, i+1) and (1,7)
      \ifnum\j=\nexti\relax
        % skip (i, i+1)
      \else
        \ifnum\i=1
          \ifnum\j=\n\relax
            % skip (1,7)
          \else
            \draw[edge] (v\i) -- (v\j);
          \fi
        \else
          \draw[edge] (v\i) -- (v\j);
        \fi
      \fi
    }
  }

  % (optional) show the removed Hamilton cycle faintly:
%  \draw[very thin, gray, dashed] (v1)--(v2)--(v3)--(v4)--(v5)--(v6)--(v7)--cycle;
\end{tikzpicture}
\caption{$K_7$ without a Hamilton cycle}\label{fig6}
\end{figure}
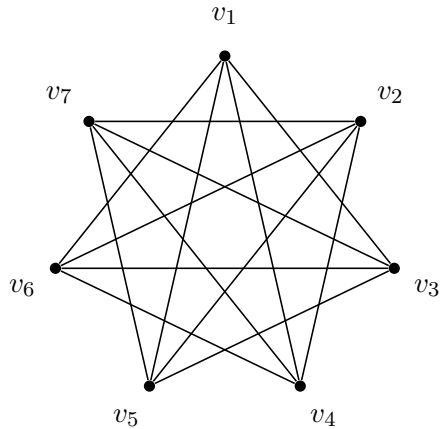

    To illustrate the tightness of the coefficient, we consider the following example:  Let $G$ be a graph obtained by deleting a Hamiltonian cycle from a $K_7$, see Fig. \ref{fig6}. We claim that
the maximum order of an odd induced subgraph, say $H$, of $G$ is exactly $2$ (any edge is an odd induced subgraph). Note that $|V(H)|$ is even by the hand-shaking lemma. If $|V(H)|=4$, there is a vertex of degree $2$ in $H$, a contradiction. If $|V(H)|=6$, there is a vertex of degree $4$ in $H$, a contradiction. So, $H$ is an edge.

\section{Conclusion}\label{sec:conc}
As we only consider the graphs without isolated vertices, so we omit saying this in the following.
In this paper, we determined the optimal constant $c$ such that every $n$-vertex graph
with maximum degree at most four contains an odd induced
subgraph of order at least $c n$.  We proved that $c = \tfrac{2}{7}$ and this value cannot be improved. Together with the
result of Berman, Wang and Wargo~\cite{BWW1997}, which established the sharp bound
$c = \tfrac{2}{5}$ for graphs of maximum degree at most three, our result completes the exact determination of the largest possible induced odd subgraph in graphs of maximum degree at most four. What is the largest value of $c$ for graphs with maximum degree five? If it is no less than $\tfrac{2}{7}$ then it would further support the possibility that $\tfrac{2}{7}$ is the optimal constant for all graphs. 

Our work fits into a broader line of research investigating the existence of large odd
induced subgraphs under structural restrictions. In particular, it complements the recent
positive resolution by Ferber and Krivelevich~\cite{FK} of a conjecture that there is a constant $c>0$ that every graph contains an odd induced subgraph of order at least $cn$. Ferber and Krivelevich's proof shows that $c\ge 10^{-4}.$
Our work parallels several optimal results known for classes such as trees~\cite{Scott1995},
graphs of treewidth at most two~\cite{hou2018}, and planar graphs with large
girth~\cite{Rao2022}. It remains a compelling direction to understand to what
extent similar sharp bounds can be obtained for other sparse families of graphs, including
graphs of bounded average degree or bounded degeneracy.

Another natural direction is related to the conjecture of Scott~\cite{Scott1992}, which states
that every $n$-vertex graph with chromatic number $\chi$ contains
an odd induced subgraph of order at least $n/(2\chi)$. This conjecture has recently been
disproved for bipartite graphs but remains open for graphs with $\chi \ge 3$; see
Wang and Wu~\cite{T2024}. It would be interesting to investigate whether the extremal
constructions arising in our proof shed further light on this problem.

\medskip

\subsection*{Acknowledgement}
JA is partially supported by the National Natural Science Foundation of China (No.12401456, No.12522117) and the Natural Science Foundation of Tianjin (No.24JCQNJC01960). QG is partially supported by China Scholarship Council (No.202406200159). YH is partially supported by China Scholarship Council (No.202406200160).


\begin{thebibliography}{99}
\bibitem{BWW1997} D.M. Berman, H. Wang, and L. Wargo, Odd induce subgraphs in graph of maximum degree three,  Australas. J. Combin. 15(1997), 81--85.
\bibitem{Caro1994} Y. Caro, On induced subgraphs with odd degrees. Discret. Math. 132 (1994), 23--28.
\bibitem{FK} A. Ferber and M. Krivelevich, Every graph contains a linearly sized induced subgraph with all degrees odd, Advances in Math. 406 (2022) 108534. 
\bibitem{hou2018} X. Hou, L. Yu, J. Li and B. Liu, Odd induced subgraphs in graphs with treewidth at most two. Graphs Combin. 34 (2018), 535--544.
\bibitem{L1979} L. Lovász, Combinatorial Problems and Exercises, North-Holland, Amsterdam (1979).
 \bibitem{Scott1995} A.J. Radclife and A.D. Scott,  Every tree contains a large induced subgraph with all degrees odd. Discret. Math. 140 (1995), 275--279.
 \bibitem{Rao2022} M. Rao, J. Hou and Q. Zeng, Odd induced subgraphs in planar graphs with large girth. Graphs Combin. 38 (2022) 105.
\bibitem{Scott1992} A.D. Scott, Large induced subgraphs with all degrees odd, Comb. Probab. Comput. 1 (1992), 335--349.



\bibitem{T2024} T. Wang and B. Wu, Maximum odd induced subgraph of a graph concerning its chromatic number, J. Graph Theory 107 (2024), 578--596. 







\end{thebibliography}
\end{document}